\documentclass[11pt]{article}
\usepackage[left=1in, right=1in, top=1in, bottom=1in, nohead]{geometry}
\usepackage{amsmath}
\usepackage{amssymb}
\usepackage{amsthm}
\usepackage{url}
\usepackage{amsfonts}
\usepackage{enumerate}
\usepackage{color}
\usepackage{verbatim}
\usepackage{multirow}
\usepackage{setspace}

\theoremstyle{definition}
\newtheorem{lem}{Lemma}[section]

\newtheorem{thm}[lem]{Theorem}
\newtheorem{cor}[lem]{Corollary}
\newtheorem{prop}[lem]{Proposition}

\newcommand{\rad}{\kappa}
\newcommand{\length}{\ell}
\newcommand{\boldh}{\mathbf{H}}
\newcommand{\boldk}{\mathbf{K}}
\newcommand{\boldl}{\mathbf{L}}

\newcommand{\boldv}{\mathbf{V}}
\newcommand{\boldc}{\mathbf{C}}
\newcommand{\boldp}{\mathbf{P}}
\newcommand{\boldx}{\mathbf{X}}
\newcommand{\boldy}{\mathbf{Y}}

\newcommand{\real}{\mathbb{R}}
\newcommand{\comp}{\mathbb{C}}
\newcommand{\disc}{\mathbb{D}}
\newcommand{\tor}{\mathcal{T}}
\newcommand{\grad}{\nabla}
\newcommand{\lap}{\Delta}
\newcommand{\uvel}{\mathbf{u}}
\newcommand{\vel}{\mathbf{v}}

\newcommand{\dis}{\mathbf{x}}
\newcommand{\traj}{\mathbf{y}}
\newcommand{\trajcomp}{y}
\newcommand{\bodyf}{\mathbf{f}}
\newcommand{\bodyg}{\mathbf{g}}
\newcommand{\eps}{\varepsilon}
\newcommand{\per}{\textrm{p}}

\begin{document}

\title{On a Linearized Problem Arising in the Navier-Stokes Flow of a Free Liquid Jet}
\author{Shaun Ceci$^1$\thanks{Corresponding author; electronic mail: {\tt cecisj@lemoyne.edu}} \, and Thomas Hagen$^2$\\ \\
$^1$Department of Mathematics, Le Moyne College\\
Syracuse, NY 13214-1399, USA\\
$^2$Department of Mathematical Sciences, The University of Memphis\\
Memphis, TN 38152-3240, USA}
\date{}
\maketitle

\begin{abstract} 
In this work, we analyze a Stokes problem arising in the study of the Navier-Stokes flow of a liquid jet. The analysis is accomplished by showing that the relevant Stokes operator accounting for a free surface  gives rise to a sectorial operator which generates an analytic semigroup of contractions. Estimates on  solutions are established using Fourier methods. The result presented is the key ingredient in a local existence and uniqueness proof for solutions of the full nonlinear problem.


\vspace{2mm}
\end{abstract}

\section{Introduction}

In this paper, we are concerned with solutions of the modified Stokes problem
\begin{align}
	\label{eq:lin_prob_main} D_t \vel - \mu \lap \vel + \grad q &= \bodyf & & \mbox{on } (0,T) \times \Omega\\
	\label{eq:lin_prob_div}\grad \cdot \vel &= 0 & & \mbox{on } (0,T) \times \Omega\\
	\label{eq:lin_prob_init}\vel(0,\cdot) &= 0 & & \mbox{on }\Omega\\
	\label{eq:lin_prob_freesurf}\mathbf{S}(\vel,q) &= 0 & & \mbox{on } (0,T) \times S_F  \\
	\label{eq:lin_prob_per} D^m_3 q |_{\Gamma_{\length}} = D^m_3 q |_{\Gamma_0}, D^k_3 \vel |_{\Gamma_{\length}} &= D^k_3 \vel |_{\Gamma_0} && \textrm{for } 0 \leq m \leq s -2, \mbox{ } 0 \leq k \leq s -1
\end{align}
for suitable initial data and  sufficiently general body forces $\bodyf$. Here $\Omega$ denotes the set
\begin{equation}
\Omega = \disc \times (0,\length),
\end{equation}
where $\length>0$ and $\disc = \left \{ (a_1,a_2) \in \real^2: a_1^2+a_2^2< \rad^2 \right \}$ for some radius $\rad>0$. We are primarily interested in thin fluid filaments (i.e., where $\rad$ is small relative to the axial period $\length$) and hence we can assume $\rad<1$.
Throughout, Cartesian coordinates in $\real^n$ will be written in the form $(a_1,\ldots, a_n)$. 
 $S_F$ denotes the portion of $\partial \Omega$ corresponding to the cylinder surface, given by
\begin{equation}
	S_F = \left \{ (a_1, a_2, a_3)\in \partial \Omega: a_1^2+a_2^2= \rad^2 \right \},
\end{equation}
$\Gamma_0$ and $\Gamma_{\length}$ denote the opposing faces of $\partial \Omega$
\begin{equation}
	\Gamma_0 = \disc \times \{0\},\qquad \Gamma_\length = \disc\times \{ \length \},
\end{equation}
and $s \geq 2$. 
The quantity $\vel$ is the (Lagrangian) fluid velocity, $q$ is the (Lagrangian) fluid pressure, $\mathbf{n}$ is the outward unit normal to $\Omega$, and $\mu>0$ is the (constant) fluid viscosity. The moving free-boundary condition is abbreviated by $\mathbf{S}(\vel,q)  = 0$ on $S_F$, where
\begin{equation}
	\mathbf{S}(\vel,q) = \Big( qn_i - \mu \sum_{j=1}^3 (D_j v_i + D_i v_j)n_j \Big)_{i=1}^3.
\end{equation} 
Note that the conditions \eqref{eq:lin_prob_per} require that solutions be periodic of period $\length$ in the $a_3$-direction.
Our objective in this work is to show that this Stokes problem allows unique solutions for given initial data and arbitrary $T>0$.

We now briefly motivate how the linear problem \eqref{eq:lin_prob_main}--\eqref{eq:lin_prob_per} arises in the study of free fluid jets. We consider the three-dimensional motion of a jet bounded by an evolving free surface under isothermal conditions and without surface tension. The fluid is assumed to be viscous, homogeneous, incompressible, and Newtonian. To model the fluid jet, the three-dimensional incompressible Navier-Stokes equations are coupled with periodic boundary conditions in the axial direction as in \cite{Teramoto} and moving free-surface boundary conditions in the radial direction:
\begin{align}
	D_t \uvel + (\uvel \cdot \grad)\uvel - \mu \lap \uvel + \grad p &= g \, \mathbf{e}_3 & & \mbox{on }\Omega(t) \label{eq:euler1} \\
	\grad \cdot \uvel &= 0  & & \mbox{on }\Omega(t) \\
	(p I - \mu (\grad \uvel + \grad \uvel^T)) \cdot \mathbf{N} &= P_0 \mathbf{N} & &  \mbox{on }S_F(t) \\
 D^m_3 p |_{\Gamma_{\length}(t)} = D^m_3 p |_{\Gamma_0(t)}, D^k_3 \uvel |_{\Gamma_{\length}(t)} &= D^k_3 \uvel |_{\Gamma_0(t)} && \textrm{for } 0 \leq m \leq s -2, \mbox{ } 0 \leq k \leq s -1  \\
	D_t \traj (t,\cdot) &= \uvel (t, \traj(t,\cdot) ) & & \mbox{on }\Omega\\
	\uvel(0,\cdot) &= \uvel_0(\cdot)& & \mbox{on }\Omega\\
	\traj(0,\cdot) &= \mathbf{I}(\cdot)& & \mbox{on }\Omega. \label{eq:euler7}
\end{align}
In this Eulerian description $\uvel$ is the fluid velocity, $p$ is the fluid pressure, $\traj$ is the fluid parcel trajectory map, $P_0$ is the (constant)
ambient pressure, and $g$ is the acceleration due to gravity. At time $t$, the fluid domain, free surface, and periodic faces are given by $\Omega(t) = \traj(t,\Omega)$, $S_F(t) = \traj(t, S_F)$, $\Gamma_0 (t)= \traj(t,\Gamma_0)$, and $\Gamma_{\length} (t) = \traj(t, \Gamma_{\length})$, respectively. $\mathbf{N}$ is the outward unit normal to $\Omega(t)$ and $\mathbf{e}_3 = (0,0,1)^T$.

The periodic boundary condition is chosen because it leads to a simpler functional setting and avoids all axial boundary layer difficulties while retaining the primary mathematical challenges of the problem. In addition, the assumption of periodicity in the axial direction has been successfully used to study physical flow phenomena  in the numerical simulation of drop dynamics for viscoelastic fluid jets \cite{li_fontelos}. 

To obtain a fixed fluid domain, it is useful to shift to a Lagrangian specification of the flow field. The problem \eqref{eq:euler1}--\eqref{eq:euler7} then becomes
\begin{align}
	D_t v_i - \mu \sum_{j,k,m=1}^3 \lambda_{j,k} D_k(\lambda_{j,m} D_m v_i) + \sum_{k=1}^3 \lambda_{i,k} D_k q &= g \delta_{3,i} & & \mbox{for }i \in \{1,2,3\} \mbox{ on } \Omega \label{eq:lag1}\\
	\sum_{j,k=1}^3 \lambda_{j,k} D_k v_j &= 0 & & \mbox{on } \Omega \\
	q n_i - \mu \sum_{j,k=1}^3 (\lambda_{j,k} D_k v_i + \lambda_{i,k} D_k v_j) n_j &= 0 & &\mbox{for } i \in \{1,2,3\} \mbox{ on } S_F 
\end{align}
\begin{align}
D^m_3 q |_{\Gamma_{\length}} = D^m_3 q |_{\Gamma_0}, D^k_3 \vel |_{\Gamma_{\length}} = D^k_3 \vel |_{\Gamma_0}, D^k_3 \dis |_{\Gamma_{\length}} &= D^k_3 \dis |_{\Gamma_0} && \textrm{for } 0 \leq m \leq s -2, \mbox{ } 0 \leq k \leq s -1  \\
 	D_t \dis &= \vel & & \mbox{on } \Omega \\
	\vel(0,\cdot) &= \uvel_0(\cdot)& & \mbox{on }\Omega \\
	\dis(0,\cdot) &= 0 & & \mbox{on }\Omega . \label{eq:lag7}
\end{align}
Here $\vel$ is the Lagrangian fluid velocity, $q $ is the difference between the Lagrangian fluid pressure and the ambient pressure, $\dis = \traj - \mathbf{I}$ is the fluid parcel displacement map, and $\delta_{i,j}$ denotes the Kronecker delta. One consequence of converting the governing equations to the Lagrangian specification is the introduction of {\em a priori} unknown quantities involving derivatives of the trajectory map $\traj$, which we denote by $\lambda_{i,j}(t,\mathbf{a}): \Omega \rightarrow \real$ where
\begin{equation}
	 \Lambda = \begin{pmatrix} \lambda_{i,j} \end{pmatrix} = (\grad \traj )^{-1} = \begin{pmatrix} D_1 \trajcomp_1 & D_1 \trajcomp_2 & D_1 \trajcomp_3 \\ D_2 \trajcomp_1 & D_2 \trajcomp_2 & D_2 \trajcomp_3 \\ D_3 \trajcomp_1 & D_3 \trajcomp_2 & D_3 \trajcomp_3 \end{pmatrix}^{-1}. 
\end{equation}
It readily follows from a continuity argument that $\dis \approx 0$ for $t\ll 1$ so that $\Lambda$ is approximately equal to the $3 \times 3$ identity matrix for small times $t$. Taking $\lambda_{i,j} = \delta_{i,j}$ in \eqref{eq:lag1}--\eqref{eq:lag7} we obtain, with the exception of the initial data, the linearized Stokes problem \eqref{eq:lin_prob_main}--\eqref{eq:lin_prob_per}. Details of the Lagrangian coordinate change for closely related problems are given by Beale \cite{beale1} and Teramoto \cite{Teramoto}.

To obtain a local-in-time solution for the full nonlinear problem \eqref{eq:lag1}--\eqref{eq:lag7}, a fixed point approach can be used, which requires unique solvability of a slightly more general version of the linearized problem discussed here. That general case, however, can be shown to reduce to the one treated in this paper. This overarching strategy of studying a modified Stokes problem in Lagrangian coordinates and its use in a fixed point argument was championed  by Beale in \cite{beale1} for a semi-infinite ``ocean'' of fluid having a free upper surface and fixed bottom. Teramoto subsequently adapted Beale's techniques to gain similar results for a free surface problem involving axisymmetric flow down the exterior of a solid vertical column of sufficiently large radius \cite{Teramoto}. 

It is important, however, to note that while the fluid jet appears similar to the problem considered in \cite{Teramoto}, there are key differences which require that we build upon the work done by Beale and Teramoto. For example, unlike the fluid domains under consideration in \cite{allain1,allain2,beale1,beale2,Teramoto}, there is no stationary surface opposite the free surface to which a Dirichlet boundary condition can be assigned. Foremost among the consequences of not having such a condition are the loss of general applicability of the Poincar\'{e} inequality and the loss of invertibility of the modified Stokes operator, which is central to the analysis. Moreover, where Teramoto is able to exploit axisymmetry and cylindrical coordinates to reduce his problem to two dimensions, the same approach introduces significant challenges in the fluid jet case since the Navier-Stokes equations in cylindrical coordinates have singular coefficients when the axis at $r=0$ is contained in the fluid domain. 

Also in contrast to \cite{allain1,allain2,beale1,beale2,Teramoto}, we utilize an elegant semigroup approach to the linearized problem. This has the benefit of immediately providing a solution to the abstract Cauchy problem associated with the problem \eqref{eq:lag1}--\eqref{eq:lag7}. We improve upon the spectral analysis of the corresponding operator $-A$ found in \cite{beale1} and show that it is, in this setting, a sectorial operator which generates an analytic semigroup of contractions. Additionally, we are able to establish explicit characterizations for both spaces in the modified Helmholtz decomposition of $(L^2(\Omega))^3$. 

This paper is organized as follows: in Section \ref{sec:prelim}, we introduce the setting of the problem and present some preliminary lemmas which adapt and extend standard results from \cite{beale1}  to fit this setting; in Section \ref{sec:spec}, we derive the relevant abstract Cauchy problem and restrict the spectrum of the underlying differential linear operator $A$ to a sector in the right half of the plane as well as provide estimates on the resolvent operator of $A$; in Section \ref{sec:semi}, we establish that $-A$ is the infinitesimal generator of an analytic semigroup of contractions and use this solve the linearized problem \eqref{eq:lin_prob_main}--\eqref{eq:lin_prob_per}.

\section{State Spaces and Estimates} \label{sec:prelim}
\setcounter{equation}{0}

To analyze equations \eqref{eq:lin_prob_main}--\eqref{eq:lin_prob_per}, we take as our initial fluid domain (the space occupied by the fluid at $t=0$) the infinite cylinder along the $a_3$-axis, 
\begin{equation}
\Omega_{\infty} = \left \{ (a_1,a_2,a_3) \in \real^3: a_1^2+a_2^2< \rad^2 \right \},
\end{equation} 
with free surface $\partial \Omega_{\infty} = \left \{ (a_1,a_2,a_3) \in \real^3: a_1^2+a_2^2=\rad^2 \right \}$. We restrict our attention to flow which is periodic in the $a_3$ direction, hence we are interested primarily in functions of the form
\begin{equation}
f = \sum_n \hat{f}_n(a_1,a_2) e^{2\pi i n a_3/ \length }  \in H^k_{\textrm{loc}}(\Omega_{\infty}) = W^{k,2}_{\textrm{loc}}(\Omega_{\infty}),
\end{equation} 
with $\hat{f}_n \in H^k(\disc)$. In practice however, we will find it more convenient to work with functions over a single period. It is natural then to interpret $a_3$-periodic functions on $\Omega_{\infty}$ as being defined on a solid torus $\tor \subset \real^3$. It should be clear that $H^k(\tor)$ is smoothly isomorphic to the space of functions of interest.

While $\tor$ is a natural choice for the domain given the periodic setting, we prefer to work in the physical space occupied by $\Omega_{\infty}$. To this end, we notice that there is a $C^{\infty}$ diffeomorphism from one period of $\Omega_{\infty}$ onto $\tor$ and consider the bounded set 
$\Omega = \disc \times (0,\length)$ with boundary $\partial \Omega = S_F \cup \Gamma_0 \cup \Gamma_{\length}$.
While the use of $\Omega$ in place of $\Omega_{\infty}$ does give rise to minor technical issues (as opposed to $\tor$) concerning the regularity of functions as one approaches the ``artificial'' corners in the boundaries, most of these problems can be dealt with by temporarily exchanging $\Omega$ for a larger subset of $\Omega_{\infty}$. As such we will occasionally find a use for the set $\Omega_1 = \disc \times (-\length, \length)$.

 Given a spatial domain $U \subset \real^3$ and a time interval $I \subset \real$, the following notational conventions are adopted for arbitrary function spaces $X(U)$ and $Y(I \times U)$:
\begin{align}
&\boldx(U) = (X(U))^3, && \boldy(I \times U) = (Y(I \times U))^3, \\
&\boldx_{\sigma}(U) = \left \{ \uvel \in \boldx(U) : \grad \cdot \uvel = 0 \right \}, && \boldy_{\sigma}(I \times U) = \left \{ \uvel \in \boldy(I \times U) : \grad \cdot \uvel = 0 \right \}, \\
&X = X(\Omega), && Y = Y( (0,T) \times \Omega), \\
&{}^0 X = \left \{ u \in X : u = 0 \mbox{ on } S_F \right \}, && {}^0 Y = \left \{ u \in Y : u = 0 \mbox{ on } S_F \right \}. 
\end{align}
Here the vector and tensor fields are equipped with the Euclidean and Frobenius norms, respectively. To keep the notation simple, we use the following rule: If a function space already has a subscript, its divergence-free subspace will be denoted by appending a $\sigma$ to the existing subscript. Our notation is thus closely aligned to the one chosen by Beale in \cite{beale1}. Spaces not following these conventions will be explicitly defined in each instance. We now introduce the spaces fundamental to this text; though each assumes $\Omega$ as its spatial domain, the extension to $\Omega_1$ is obvious. For the set of functions on $\Omega$ whose $a_3$-periodic extensions are continuously differentiable and bounded on $\Omega_{\infty}$ we simply take $C^k_{\per} = \left \{ u|_{\Omega}: u = \sum_{n=-\infty}^{\infty} \hat{u}_n(a_1,a_2) e^{2 \pi i n a_3/ \length } \in C^k \left( \overline{\Omega_{\infty}} \right) \right \}$. Note that $\hat{u}_n \in C^k(\disc)$ necessarily. Similarly, we define $C^{\infty}_{\per}$ (or $C^{k,\alpha}_{\per}$) to be the set of all such functions which are bounded and smooth (or H\"{o}lder continuous with exponent $\alpha$) on $\Omega_{\infty}$. It is clear that the following space is isomorphic to $H^k(\tor)$:
\begin{equation}
	H_{\per}^k = \displaystyle \left \{ u= \sum_{n=-\infty}^{\infty} \hat{u}_n(a_1,a_2) e^{2 \pi i n a_3/ \length } \in H^k(\Omega): \hat{u}_n \in H^k(\disc) \mbox{ and } (u,u)_{H_{\per}^k} < \infty \right \}
\end{equation}
for  $k \in \mathbb{N}_0$, where 
\begin{equation}
	( u,v )_{H_{\per}^k} = \displaystyle \sum_{n=-\infty}^{\infty} \hspace*{.4em} \sum_{m=0}^{k} \frac{(2 \pi n)^{2m}}{\length^{2m-1}} ( \hat{u}_n, \hat{v}_n )_{H^{k-m}(\disc)} \qquad \mbox{and} \qquad \| u \|_{H^k_{\per}} = \sqrt{(u,u)_{H^k_{\per}}}. 
\end{equation}
The norms $\| \cdot \|_{H^k_{\per}}$ and $\| \cdot \|_{H^k}$ are equivalent norms on $H^k_{\per}$ which are actually equal for $k \in \{ 0,1 \}$. Here $\| \cdot \|_{H^k}$ denotes the standard norm on $H^k(\Omega)$. It then readily follows that $H^0_{\per} = L^2$ and
\begin{equation} \label{eq:Hkp_char}
H^k_{\per} = \{ f \in H^k: D_3^j f|_{\Gamma_{\length}} = D_3^j f|_{\Gamma_0} \mbox{ for all } 0 \leq j \leq k-1\}
\end{equation}
for $k \geq 1$. We define $H_{\per}^s$, $s \in \real^+$, using complex interpolation. Note that, throughout the text, we typically use $r$ and $s$ to denote non-integer regularity and $k$ and $m$ when we restrict ourselves to integer regularity. 

Instead of the standard Helmholtz decomposition of $\boldl^2$, we take our lead from \cite{beale1} and pursue something slightly different. In particular, to incorporate $a_3$-periodicity along with the divergence-free condition into the auxiliary space, we choose our decomposition so that $\boldl^2$ can be projected onto $ \boldh^0_{\per \sigma}$. For convenience, we set $\boldp^s = \boldh^s_{\per \sigma}$ and introduce the space
\begin{equation}
\boldv^s = \{ \vel \in \boldp^s : \mathbf{S}_{\textrm{tan}} (\vel) = 0 \mbox{ on } S_F \},
\end{equation}
where $\mathbf{S}_{\textrm{tan}} = \mathbf{S} - (\mathbf{S} \cdot \mathbf{n}) \mathbf{n} $ is the \emph{tangential part} of $\mathbf{S}$. Finally, to incorporate regularity with respect to time we define the space
\begin{equation}
K_{\per}^s(I \times \Omega) = H^{s/2} (I ; H^0_{\per}) \cap H^0(I;H^s_{\per}).
\end{equation}

In contrast to \cite{beale1}, we now provide an explicit characterization of the orthogonal complement $(\boldp^0)^{\perp}$ arising in our Helmholtz decomposition of $\boldl^2$. This result will prove important later on.

\begin{prop} \label{prop:orthog} The orthogonal complement of $\boldp^0$ in $\boldl^2$ has the characterization $(\boldp^0)^{\perp} = \{ \grad q : q \in {}^0 H^1_{\per} \}$. 
\end{prop}
\begin{proof}
Let $\boldy = \{ \grad q : q \in {}^0 H^1_{\per} \}$. It is sufficient to show two things: (i) $\boldy$ is closed in $\boldl^2$ so that $\boldy = (\boldy^{\perp})^{ \perp}$, and (ii) $\boldp^0 = \boldy^{\perp}$. In order to prove (i), we will first need to show that the orthogonal complement of $\boldx = \overline{ {}^0 \boldc^{\infty}_{{\per} \sigma}}^{\| \cdot \|_{\boldl^2}}$ in $\boldl^2$ has the characterization $\boldx^{\perp} = \{ \grad q : q \in H^1_{\per} \}$. Let $q \in H^1_{\per}, \uvel \in \boldx$. There exist $\uvel_k \in {}^0 \boldc^{\infty}_{\per \sigma}$ such that $\uvel_k \rightarrow \uvel$ in $\boldl^2$. Integration by parts yields
   	\begin{equation}
   		(\grad q, \uvel)_{\boldl^2} = \lim_{k \rightarrow \infty} (\grad q, \uvel_k)_{\boldl^2} = \lim_{k \rightarrow \infty} \int_{\Gamma_{\length}} \! q \overline{ \uvel_k } \cdot \mathbf{e}_3 + \int_{\Gamma_0} \! q \overline{ \uvel_k } \cdot (-\mathbf{e}_3) = 0.
	\end{equation}
	Thus $\grad q \in \boldx^{\perp}$. Conversely, let $\mathbf{w} \in \boldx^\perp$. Then, in particular, $(\mathbf{w},\uvel)_{\boldl^2} = 0$ for all $\uvel\in \boldc^{\infty}_{\textrm{c} \sigma}$. Thus there exists $p \in H^1$ such that $\mathbf{w} = \grad p$, see \cite[pp. 10--11]{temam2}. Now consider
	\begin{equation}
		\uvel = \begin{pmatrix} 0 \\ 0 \\ u(a_1,a_2) \end{pmatrix} \in {}^0 \boldc^{\infty}_{\per \sigma},
	\end{equation}
	where $ u \in C^{\infty}_{\textrm{c}} ( \disc )$ is arbitrary. Then, applying integration by parts, we obtain
	\begin{align}
		0 = (\mathbf{w}, \uvel)_{\boldl^2}
		&= ( \grad p , \uvel )_{\boldl^2} \\
		&= \int_{\Gamma_{\length}} \! p \overline{ \uvel } \cdot \mathbf{e}_3 + \int_{\Gamma_0} \! p \overline{ \uvel } \cdot (-\mathbf{e}_3) \\
		&= \int_{\Gamma_{\length}} \! p \overline{ u } - \int_{\Gamma_0} \! p \overline{ u } \\
		&= \left(  p|_{\Gamma_{\length}} - p|_{\Gamma_0} , u \right)_{L^2(\disc)}.
	\end{align}
	Since $u$ is an arbitrary element of a dense subset of $L^2(\disc)$ (see \cite[p. 13]{mazja}), this implies that $p|_{\Gamma_{\length}} = p|_{\Gamma_0} $ on $L^2(\disc)$. Hence $p \in H^1_{\per}$ by \eqref{eq:Hkp_char}.
	
With this characterization in hand, we can now prove (i). Let $q_k \in {}^0 H^1_{\per}$ such that $\grad q_k \rightarrow \bodyf \in \boldl^2$. Since $\grad q_k \in \boldx^{\perp}$, we have
\begin{equation}(\bodyf, \uvel)_{\boldl^2} = \lim_{k \rightarrow \infty} (\grad q_k, \uvel)_{\boldl^2} = 0
\end{equation}
for all $\uvel \in \boldx$. Hence $\bodyf \in \boldx^{\perp}$ and so there exists $p \in H^1_{\per}$ such that $\bodyf = \grad p$. Notice that for $n \neq 0$
\begin{align}\| (\hat{q}_k)_n - \hat{p}_n \|^2_{H^1(\disc)} &\leq \length^2 \sum_n \left( \left( \frac{2 \pi n}{\length} \right)^2 \| (\hat{q}_k)_n - \hat{p}_n \|^2_{H^1(\disc)} + \sum_{j=1}^2 \| D_j( (\hat{q}_k)_n - \hat{p}_n) \|^2_{L^2(\disc)} \right)\\
&\leq \length^2 \| \grad (q_k - p) \|^2_{\boldl^2}.
\end{align}
Thus $(\hat{q}_k)_n \rightarrow \hat{p}_n$ in $H^1(\disc)$. Since $(\hat{q}_k)_n \in H^1_0(\disc)$, a closed subspace of $H^1(\disc)$, we obtain $\hat{p}_n \in H^1_0(\disc)$ for $n \neq 0$. For $n = 0$, applying the standard Poincar\'{e} inequality yields a constant $C>0$ such that 
\begin{equation}\| (\hat{q}_k)_0 - (\hat{q}_m)_0 \|^2_{H^1(\disc)} \leq C \| \grad((\hat{q}_k)_0 - (\hat{q}_m)_0) \|^2_{(L^2(\disc))^2} \leq C \length^2 \| \grad (q_k - q_m) \|^2_{\boldl^2}
\end{equation}
which implies that $(\hat{q}_k)_0$ converges in $H^1_0(\disc)$. Moreover, the limit is necessarily $\hat{p}_0+ \lambda$, for some $\lambda \in \real$, since it is readily seen that $(\hat{q}_k)_0$ converges to this limit in the weaker $L^2$-norm. Thus $\bodyf = \grad q$ where $q = p + \lambda \in {}^0 H^1_{\per}$. Hence $\boldy$ is closed in $\boldl^2$.

Finally, we show (ii). Let $\uvel \in \boldy^{\perp}$ and $\varphi \in C^{\infty}_{\textrm{c}}$. Then
\begin{equation}0 = (\grad \varphi, \uvel)_{\boldl^2} = - \int_{\Omega} \varphi (\overline{\grad \cdot \uvel}) .
\end{equation}
Hence $\grad \cdot \uvel$ acts as a bounded linear functional on $C^{\infty}_{\textrm{c}}$ and can be extended to all of $L^2$ by density. This unique operator must be the zero functional and thus $\uvel \in \boldp^0$. Conversely, let $\vel \in \boldp^0$. Since $\boldl^2 = \boldy \oplus \boldy^{\perp}$, there are $q \in {}^0 H^1_{\per}$ and $\tilde{\vel} \in \boldy^{\perp}$ such that $\vel = \tilde{\vel} + \grad q$. Taking the divergence of both sides of this equation yields $\lap q = 0$ and, by Lax-Milgram, $q$ must be the unique solution of this equation in ${}^0 H^1_{\per}$. Thus $q=0$ and $\vel = \tilde{\vel} \in \boldy^{\perp}$. Thus $\boldp^0 = \boldy^{\perp}$ and the claim follows.

\end{proof}

The following are analogous to results in \cite{beale1} and are readily shown using similar techniques. They are given here explicitly for the reader's covenience.
\begin{prop}  \label{prop:P} Let $P$ be the orthogonal projection of $\boldl^2$ onto $\boldp^0$. 
\begin{enumerate}[(1)]
\item For $s \geq 0$, we have $P\boldh^s_{\per} = \boldp^s$ and $P|_{\boldh^s_{\per}} : \boldh^s_{\per} \rightarrow \boldp^s$ is bounded.
\item $P|_{\boldk^s_{\per}} : \boldk^s_{\per} \rightarrow \boldk^s_{\per}$ is bounded with norm bounded independent of $T$.
\item Suppose $s \geq 1 $. If $f \in H^s_{\per}$, then there is a unique $\tilde{f} \in H^s_{\per}$ such that 
\begin{equation}
P(\grad f) = \grad \tilde{f}, \qquad f |_{S_F}  = \tilde{f} |_{S_F}, \qquad \mbox{and} \qquad \lap \tilde{f} = 0 .
\end{equation}
\end{enumerate}
\end{prop}

%\begin{proof}
%(i) follows using standard methods such as those in \cite{temam2}. (ii), (iii) are analogous to results in \cite{beale1}.
%\end{proof}

We will see that, just as in \cite{beale1}, many crucial quantities can be cast as solutions of a particular problem involving Laplace's equation. Adapting the boundary conditions to reflect periodicity and the absence of a fixed bottom surface, the relevant problem in our setting takes the form
\begin{equation}
\lap u = f \quad \mbox{in } \Omega, \qquad \qquad u = 0 \quad \mbox{on } S_F, \qquad \qquad  D^k_3 u |_{\Gamma_{\length}} = D^k_3 u |_{\Gamma_0} \quad \textrm{for } k \in \{ 0 ,1 \}, \label{eq:lap_equ}
\end{equation}
where $f \in H^{s-2}_{\per}$ is given. The following result demonstrates that \eqref{eq:lap_equ} has a unique solution and provides an estimate for it in terms of the inhomogeneity $f$. We note that the provided proof does not draw from the corresponding proof in \cite{beale1}.

\begin{prop} \label{prop:uniquesoln} For $f \in H^{s-2}_{\per}, s \geq 2$, there is a unique solution $u \in {}^0 H^s_{\per}$ of $\lap u = f$ on $\Omega$. Additionally, there exists $C>0$, independent of $f$, such that $\| u \|_{H^s_{\per}} \leq C \| f \|_{H^{s-2}_{\per}}$.
\end{prop}

\begin{proof} Let $f = \sum_n \hat{f}_n e^{2 \pi i n a_3/ \length}$. We first consider the boundary-value problem, $L_n u = -\hat{f}_n$ on $\disc$ with $u=0$ on $\partial \disc$, where $L_n$ and its associated sesquilinear form ($B_n : H^1_0 (\disc) \times H^1_0 (\disc) \rightarrow \mathbb{C}$) are given by 
\begin{align}L_n u &= - \lap u + \left( \frac{2 \pi n}{\length} \right)^2 u\\
B_n[u,v] &= (\grad u, \grad v)_{\boldl^2(\disc)} + \left( \frac{2 \pi n}{\length} \right)^2 (u,v)_{L^2(\disc)}.
\end{align}
Clearly,  $B_n$ is continuous and coercive on $H^1_0 (\disc)$, thus we can apply Lax-Milgram to obtain a unique weak solution, $\hat{u}_n \in H^1_0 (\disc)$. The construction $u = \sum_n \hat{u}_n e^{2 \pi i n a_3/\length}$ is then our candidate for the solution of the boundary-value problem in $\Omega$. We now restrict our discussion to the case when $s=k \in \mathbb{Z}$. Given the regularity of $\partial \disc$ we can immediately conclude that each $\hat{u}_n \in H^k(\disc)$ is a strong solution. Our goal is to show that $u \in H^k_{\per}$. First we obtain some preliminary estimates for $\hat{u}_n$ where $n \neq 0$:
\begin{align}
B_n[\hat{u}_n,\hat{u}_n] &= (-\hat{f}_n,\hat{u}_n)_{L^2(\disc)}\\
\| \grad \hat{u}_n \|^2_{\boldl^2(\disc)} + \left( \frac{2 \pi n}{\length} \right)^2 \| \hat{u}_n \|^2_{L^2(\disc)} &\leq \| \hat{f}_n \|_{L^2(\disc)} \| \hat{u}_n \|_{L^2(\disc)}\\
\| \hat{u}_n \|_{L^2(\disc)} &\leq \left( \frac{\length}{2 \pi n} \right)^2 \| \hat{f}_n \|_{L^2(\disc)} 
\end{align}

Notice that from this estimate we can conclude
\begin{equation}\sum_n \left( \frac{2 \pi n}{\length} \right)^{2k} \| \hat{u}_n \|^2_{L^2(\disc)} \leq \sum_n \left( \frac{2 \pi n}{\length} \right)^{2(k-2)} \| \hat{f}_n \|^2_{L^2(\disc)} \leq \| f \|^2_{H^{k-2}_{\per}}< \infty.
\end{equation}
This gives us an estimate on the lowest order terms in the $H^k_{\per}$-norm. For the highest order terms, standard elliptic regularity theory (e.g., see \cite{Evans}, p. 323) provides an estimate of the form
\begin{equation}\| \hat{u}_n \|_{H^k(\disc)} \leq C_1 \| \hat{f}_n \|_{H^{k-2}(\disc)},
\end{equation} 
though the constant $C_1$ here generally depends on the coefficients (and hence $n$) of $L_n$. However, upon closer inspection of the proof of this result (e.g., in \cite{Evans}), we observe that our above estimates on $\| \hat{u}_n \|_{L^2(\disc)}$  can be used in place of the usual $L^{\infty}$-estimate on the coefficient $(2 \pi n/ \length)^2$ of $L_n$. In consequence this allows $C_1$ to be chosen independently of $n$. Therefore
\begin{equation}
	\sum_n \| \hat{u}_n \|^2_{H^k(\disc)} \leq C^2_1 \sum_n \| \hat{f}_n \|^2_{H^{k-2}(\disc)} \leq C^2_1 \| f \|^2_{H^{k-2}_{\per}(\Omega)}< \infty.
\end{equation}
Finally, we must show that the intermediate order terms in the $H^k_{\per}$-norm are also summable. Exploiting complex interpolation between $H^0(\disc)$ and $H^k(\disc)$ and Young's inequality, we obtain for each $0<m<k$
\begin{align}
	\sum_n \left( \frac{2 \pi n}{\length} \right)^{2m} \| \hat{u}_n \|^2_{H^{k-m}(\disc)} &\leq \sum_n \left( \frac{2 \pi n}{\length} \right)^{2m} \left( \| \hat{u}_n \|^{m/k}_{L^2(\disc)} \| \hat{u}_n \|^{1-m/k}_{H^k(\disc)} \right)^2\\
	&\leq C_2 \sum_n \left( \frac{2 \pi n}{\length} \right)^{2m(k-2)/k} \| \hat{f}_n \|^{2m/k}_{L^2(\disc)} \| \hat{f}_n \|^{2(k-m)/k}_{H^{k-2}(\disc)}
\end{align}
\begin{align}
	\makebox[3cm]{\ }&\leq C_2 \sum_n \left( \frac{m}{k} \left[ \left( \frac{2 \pi n}{\length} \right)^{2m(k-2)/k} \| \hat{f}_n \|^{2m/k}_{L^2(\disc)} \right]^{k/m} \right. \\
	&\hspace*{10em} \left. + \frac{k-m}{k} \left[ \| \hat{f}_n \|^{2(k-m)/k}_{H^{k-2}(\disc)} \right]^{k/(k-m)} \right) \\
	&= C_2 \sum_n \frac{m}{k} \left( \frac{2 \pi n}{\length} \right)^{2(k-2)} \| \hat{f}_n \|^2_{L^2(\disc)} + \frac{k-m}{k} \| \hat{f}_n \|^2_{H^{k-2}(\disc)}\\
	&\leq C_3 \| f \|^2_{H^{k-2}}.
\end{align} 
Thus $u \in {}^0 H^k_{\per}$ with $\| u \|^2_{H^k_{\per}} \leq C_4(k+1) \| f \|^2_{H^{k-2}_{\per}}$, which completes the proof for integer values of $s$. Interpolation then provides the remaining cases. 
\end{proof}

\section{A Spectral Result} \label{sec:spec}
\setcounter{equation}{0}

Our first goal is to use the modified Helmholtz projection $P$ to rewrite the problem (\ref{eq:lin_prob_main})--(\ref{eq:lin_prob_per}) in a variational form which has the velocity as its only unknown. First we notice that for any solution $(\vel, q)$ of the problem, \eqref{eq:lin_prob_div} implies $\vel(t) \in \boldp^0$ for each $t$. Since it is readily seen that $P$ commutes with $D_t$, applying $P$ to (\ref{eq:lin_prob_main}) yields
\begin{equation} 
	D_t \vel - \mu P \lap \vel + \grad q_1 = P \bodyf,
\end{equation}
where $\grad q_1 = P \grad q$ (with $\lap q_1 = 0$ on $\Omega$ and $q_1 = q$ on $S_F$) by Proposition \ref{prop:P}(3). This application of $P$ removes the indeterminacy of the pressure term in the sense that the value of $q_1$ is uniquely determined by $\vel$. 

\begin{lem} \label{lem:Q} Suppose $s \geq 2$ and $(\vel,q) \in \boldh^s_{\per} \times H^{s-1}_{\per}$ satisfies \eqref{eq:lin_prob_freesurf}. Then there exists a bounded linear operator $Q: \boldh^s_{\per} \rightarrow H^{s-1}_{\per}$ mapping $\vel \mapsto q_1$ where $q_1$ is the function provided by Proposition \ref{prop:P}(3) with $\grad q_1 = P \grad q$. 
\end{lem}
\begin{proof} To see this, we use the fact that $q$ and $q_1$ agree on the free surface and observe that \eqref{eq:lin_prob_freesurf} implies that the normal component of $\mathbf{S}(\vel,q)$ must vanish on $S_F$. Hence $\mathbf{S}(\vel,q_1) \cdot \mathbf{n} = 0$ implies $q_1 = 2 \mu \rad^{-2} \sum_{i,j=1}^2 a_i a_j D_j v_i$ on $S_F$. Given $\vel \in \boldh^s_{\per}$, we note that $f = 2 \mu \rad^{-2} \sum_{i,j=1}^2 a_i a_j D_j v_i  \in H^{s-1}_{\per}$. For $s=2$, we can apply Lax-Milgram to obtain the existence of a unique weak solution $q_1 \in H^1_{\per}$ of the problem
		\begin{align}
			\lap q_1 &= 0 & & \mbox{on } \Omega\\
			q_1 &= f & & \mbox{on } S_F\\
			D^k_3 q_1 |_{\Gamma_{\length}} &= D^k_3 q_1 |_{\Gamma_0} & & \mbox{for } k \in \{ 0, 1 \}
		\end{align}
with $\| q_1 \|_{H^1_{\per}} \leq C_1 \| \vel \|_{\boldh^2_{\per}}$ where $C_1>0$ is independent of $\vel$. For $s \geq 3$, we consider the problem
		\begin{align}
			\lap \phi &= -\lap f & & \mbox{on } \Omega\\
			\phi &= 0 & & \mbox{on } S_F\\
			D^k_3 \phi |_{\Gamma_{\length}} &= D^k_3 \phi |_{\Gamma_0} & & \mbox{for } k \in \{0, 1 \}
		\end{align}
which has a unique solution $\phi \in {}^0H^{s-1}_{\per}$, by Proposition \ref{prop:uniquesoln}, satisfying $\| \phi \|_{H^{s-1}_{\per}} \leq C_2 \| \lap f \|_{H^{s-3}_{\per}}$ for some $C_2>0$ which is independent of $\vel$. Finally, we set $q_1 = \phi + f \in H^{s-1}_{\per}$ and observe that $\| q_1 \|_{H^{s-1}_{\per}} \leq C_2 \| \lap f \|_{H^{s-3}_{\per}} + \| f \|_{H^{s-1}_{\per}} \leq C_3 \| \vel \|_{\boldh^s_{\per}}$. Interpolation now yields the claim for the remaining values of $s$. It readily follows that the constructed operator is linear in $\vel$.
\end{proof}

We now take the general approach used in semigroup theory by treating \eqref{eq:lin_prob_main} as an abstract ordinary differential equation with respect to time whose solution is, for each value of $t$, an element of the appropriate function space ($\boldv^s$) on $\Omega$. If we define an operator $A: \boldv^s \rightarrow \boldp^{s-2}$ by
\begin{equation}
A \vel = - \mu P \lap \vel + \grad Q \vel,
\end{equation}
the problem \eqref{eq:lin_prob_main}--\eqref{eq:lin_prob_per} takes on the form
\begin{align}\dot{\vel} + A \vel &= P \bodyf  & & \mbox{on } (0,T) \times \Omega  \label{eq:var_lin_prob1}\\
\vel(0,\cdot) &= 0 & & \mbox{on }\Omega \label{eq:var_lin_prob2}.
\end{align}
Notice that \eqref{eq:lin_prob_freesurf} is satisfied since our construction of $Q$ ensures that the normal component of $\mathbf{S}(\vel,q)$ will vanish on $S_F$. The operator $A$ is a modification of the standard Stokes operator.

We tackle the matter of determining the spectrum of $-A$ first. Unfortunately, in contrast to the problems treated in \cite{allain1, allain2, beale1, beale2,Teramoto}, $A$ is not injective with our boundary conditions (implying that $0$ lies in the spectrum of $A$) since $A(\vel+\mathbf{c})=A\vel$ for any constant vector $\mathbf{c}$. This, combined with the inability to apply the Poincar\'{e} inequality in general, makes the problem of determining the spectrum more challenging here than in the aforementioned cases. Restricting the spectrum of $A$ to a sector in the right half of the plane and providing estimates on the resolvent operator, the following theorem is a key result of this work.

We will show that $\rho(A)$ contains all $\lambda\in \comp $ such that $| \textrm{Im}(\lambda) | > \textrm{Re}(\lambda)$. Given $\bodyg \in \boldp^{s-2}$ and $\lambda$ with $| \textrm{Im}(\lambda) | > \textrm{Re}(\lambda)$, we find a unique solution $(\vel , q) \in \boldv^s \times H^{s-1}_{\per}$ of the problem 
\begin{equation} \label{eq:stat_lin}
-\mu \lap \vel - \lambda \vel + \grad q = \bodyg
\end{equation}
along with \eqref{eq:lin_prob_div}, \eqref{eq:lin_prob_freesurf}, and \eqref{eq:lin_prob_per}. To see that this is equivalent to the statement about $\rho(A)$, suppose that there exists $\vel \in \boldv^s$ such that $(A- \lambda I) \vel =  \bodyg$. Using our decomposition of $\boldl^2$, we find $q_0 \in {}^0H^1_{\per}$ such that $\grad q_0 = \mu (I-P) \lap \vel$. Setting $q = Q \vel + q_0$, we obtain \eqref{eq:stat_lin}. It is now straightforward to verify that $(\vel, q)$ also satisfies \eqref{eq:lin_prob_div}, \eqref{eq:lin_prob_freesurf}, and \eqref{eq:lin_prob_per}. Conversely, given a solution $(\vel,q)$ of the stationary problem we can apply $P$ to \eqref{eq:stat_lin} to obtain $(A - \lambda I) \vel = \bodyg$. Hence $(A - \lambda I) \vel = \bodyg$ has a unique solution $\vel$ if and only if the problem \eqref{eq:lin_prob_div}, \eqref{eq:lin_prob_freesurf}, \eqref{eq:lin_prob_per}, \eqref{eq:stat_lin} has a unique solution $(\vel,q)$.

\begin{thm} \label{thm:spectrum_weak} For $\bodyg \in \boldp^{s-2}$, the problem \eqref{eq:lin_prob_div}, \eqref{eq:lin_prob_freesurf}, \eqref{eq:lin_prob_per}, \eqref{eq:stat_lin} has a unique weak solution $(\vel,q) \in \boldp^1 \times L^2$
for each $\lambda\in \comp $ such that $| \textrm{Im}(\lambda) | > \textrm{Re}(\lambda)$.
\end{thm}

\begin{proof} Notice that for any $\vel \in \boldp^1$, $\vel$ satisfies \eqref{eq:lin_prob_div} and the portion of \eqref{eq:lin_prob_per} referring to the velocity. The free surface condition \eqref{eq:lin_prob_freesurf} is not necessarily satisfied though and will need to be incorporated into a variational formulation directly. To that end, let us consider the sesquilinear form $\langle \cdot, \cdot \rangle : \boldp^1 \times \boldp^1 \rightarrow \mathbb{C}$ defined by 
\begin{equation} \label{eq:sesqui} \langle \vel , \uvel \rangle = - \lambda ( \vel, \uvel)_{\boldl^2} + \frac{\mu}{2} \sum_{i,j=1}^3 \int_{\Omega}  (D_j v_i + D_i v_j)(D_j \bar{u}_i + D_i \bar{u}_j) . \end{equation}
Now suppose $\uvel \in \boldh^1_{\per} , \vel \in \boldp^2$, $q \in H^1_{\per}$ and observe that 
\begin{align}
\int_{\Omega} (-\mu \lap \vel - \lambda \vel + \grad q) \cdot \overline{\uvel} &= - \mu \left( \int_{\Omega} \sum_i \lap v_i \overline{u}_i \right) - \lambda ( \vel, \uvel)_{\boldl^2} + \int_{\partial \Omega} q (\overline{\uvel} \cdot \mathbf{n}) - \int_{\Omega} q \grad \cdot \overline{\uvel} \\
&= \mu \sum_{i,j} \int_{\Omega} D_j \overline{u}_i D_j v_i  + \sum_i \int_{\partial \Omega} \overline{u}_i ( q n_i - \mu \sum_j D_j v_i n_j )  \nonumber \\
& \qquad \qquad - \lambda ( \vel, \uvel)_{\boldl^2} - \int_{\Omega} q \grad \cdot \overline{\uvel}
\end{align}
\begin{align}
&= \langle \vel , \uvel \rangle + \int_{\partial \Omega} \mathbf{S} (\vel , q) \cdot \overline{\uvel}  + \mu \sum_{i,j} \int_{\partial \Omega} \overline{u}_i D_i v_j n_j  \nonumber \\
& \qquad \qquad - \int_{\Omega} D_j \overline{u}_i D_i v_j - \int_{\Omega} q \grad \cdot \overline{\uvel} \\
&= \langle \vel , \uvel \rangle + \int_{\partial \Omega} \mathbf{S} (\vel , q) \cdot \overline{\uvel} - \int_{\Omega} q \grad \cdot \overline{\uvel} + \mu \sum_i \int_{\Omega} \overline{u}_i D_i (\grad \cdot \vel)  \\
&= \langle \vel , \uvel \rangle + \int_{\partial \Omega} \mathbf{S} (\vel , q) \cdot \overline{\uvel} - \int_{\Omega} q \grad \cdot \overline{\uvel}. \label{eq:weak_form}
\end{align}
Notice that the pair $(\vel, q)$ currently satisfies \eqref{eq:lin_prob_div} and \eqref{eq:lin_prob_per}. If we suppose that $(\vel,q)$ additionally satisfies \eqref{eq:stat_lin} and \eqref{eq:lin_prob_freesurf}, then we obtain
\begin{equation}
(\bodyg, \uvel)_{\boldl^2} = \langle \vel , \uvel \rangle + \int_{\Gamma_{\length}} \mathbf{S} (\vel , q) \cdot \overline{\uvel} + \int_{\Gamma_0} \mathbf{S} (\vel , q) \cdot \overline{\uvel} = \langle \vel , \uvel \rangle
\end{equation}
for all $\uvel \in \boldp^1$. Thus $ \langle \vel , \uvel \rangle = (\bodyg, \uvel)_{\boldl^2} $ can be seen as a weak formulation of the full problem which does not involve $q$. In an effort to apply Lax-Milgram, we verify that the sesquilinear form is both continuous and coercive. Applying H\"{o}lder, $| \langle \vel, \uvel \rangle | \leq  C \| \vel \|_{\boldh^1_{\per}} \| \uvel \|_{\boldh^1_{\per}}$, where $C>0$ depends on $\mu$ and $\lambda$. Hence the sesquilinear form is continuous. That it is also coercive follows from Korn's inequality:
\begin{align}
| \langle \vel, \vel \rangle |^2 &= \left( \frac{\mu}{2} \sum_{i,j=1}^3 \int_{\Omega} | D_j v_i + D_i v_j |^2 - \mathrm{Re}(\lambda)\| \vel \|^2_{\boldl^2} \right)^2 + \left( \mathrm{Im}(\lambda) \| \vel \|^2_{\boldl^2} \right)^2 \label{eq:coercive}  \\
&\geq \frac{1}{2} \left( \frac{\mu}{2} \sum_{i,j=1}^3 \int_{\Omega} | D_j v_i + D_i v_j |^2  - \mathrm{Re}(\lambda)\| \vel \|^2_{\boldl^2}  + | \mathrm{Im}(\lambda)| \cdot \| \vel \|^2_{\boldl^2} \right)^2 \\
& \geq C \| \vel \|^4_{\boldh^1_{\per}}. 
\end{align}
For $\mathrm{Re}(\lambda) \geq 0$, line \eqref{eq:coercive} implies
\begin{equation}
| \langle \vel, \vel \rangle |^2 \geq \mathrm{Im}(\lambda)^2 \| \vel \|^4_{\boldl^2} \geq  \frac{1}{2} |\lambda|^2 \| \vel \|^4_{\boldl^2}. \label{eq:sequi_est}
\end{equation}
Moreover, the same estimate can be obtained for $\mathrm{Re}(\lambda) < 0$ since the right-hand side of line \eqref{eq:coercive} then expands to something of the form $\phi + | \lambda |^2 \| \vel \|^4_{\boldl^2}$ where $\phi \geq 0$. Since the sesquilinear form satisfies the conditions of Lax-Milgram, we obtain a unique weak solution $\vel \in \boldp^1$ of \eqref{eq:lin_prob_div}, \eqref{eq:lin_prob_freesurf}, \eqref{eq:lin_prob_per}, \eqref{eq:stat_lin} along with the estimate
\begin{equation}
\| \vel \|_{\boldh^1_{\per}} \leq C \| \bodyg \|_{\boldl^2}.
\end{equation}
We now seek an associated pressure, $q$,  of $\vel$. Recall that an associated pressure need only satisfy \eqref{eq:stat_lin} in the sense of distributions (i.e., when tested against arbitrary $\uvel \in \boldc^{\infty}_c$). As with the velocity, we begin by finding a weak formulation for the pressure. Notice that, for $q \in H^1_{\per}$ with $(\vel,q)$ satisfying \eqref{eq:lin_prob_freesurf}, we obtain from \eqref{eq:weak_form} that
\begin{align}
\int_{\Omega} q \grad \cdot \overline{\uvel} = \langle \vel, \uvel \rangle - (\bodyg, \uvel)_{\boldl^2} \label{eq:weak_form1}
\end{align}
for all $\uvel \in \boldh^1_{\per}$. Using continuity of the sesquilinear form we obtain immediately that the right-hand side is a bounded linear functional in $\overline{\uvel}$, $\mathbf{F}: \boldc^{\infty}_{\textrm{c}} \rightarrow \mathbb{C}$, which vanishes when $\grad \cdot \overline{\uvel} = 0 $. From \cite{Sohr}, we know there is a unique $\tilde{q} \in L^2$ such that 
\begin{equation}
\label{eq:q_weak_form} \mathbf{F} = \grad \tilde{q} \qquad \mbox{and} \qquad \int_{\Omega} \tilde{q} = 0.
\end{equation}
It is now straightforward to verify that $q = - \tilde{q} $ satisfies \eqref{eq:stat_lin} in the distributional sense and hence is an associated pressure of $\vel$. It is uniquely determined under the additional condition $\int_{\Omega} q = 0$, but otherwise is unique only up to a constant.
\end{proof}

The proof of the preceding theorem allows us to draw the following conclusion about the case $\lambda=0$.

\begin{cor}\label{0sol}
For $\bodyg =0$, $(\vel,q) \in \boldp^1 \times L^2$ is a weak solution of the problem \eqref{eq:lin_prob_div}, \eqref{eq:lin_prob_freesurf}, \eqref{eq:lin_prob_per}, \eqref{eq:stat_lin} with $\lambda=0$ if and only if $\vel$ and $q$ are constant.
\end{cor}

\begin{proof}
	Reasoning as in \eqref{eq:coercive}, we have 
\begin{equation}
	0 =  | \langle \vel, \vel \rangle | =  \frac{\mu}{2} \sum_{i,j=1}^3 \int_{\Omega} | D_j v_i + D_i v_j |^2 \geq 2\mu \sum_{i=1}^3 \int_{\Omega} | D_j v_i |^2 .
\end{equation}
Hence $\vel$ and {\em a fortiori} $q$ are constant.
\end{proof}

\begin{thm} \label{thm:resolve_inv} Let $s \geq 2$. Then $\sigma(A) \subset \{ \lambda \in \mathbb{C} : | \textrm{Im}(\lambda) | \leq \textrm{Re}(\lambda) \} $. Moreover, for  $\lambda$ with $| \textrm{Im}(\lambda) | > \textrm{Re}(\lambda)$ and $| \lambda | \geq \eps > 0$ the resolvent operator $R(\lambda; A) = (A-\lambda I)^{-1}$ satisfies
\begin{equation} \label{eq:resolve_ineq} 
\| R(\lambda; A) \bodyg \|_{\boldh_p^s} \leq C (\| \bodyg \|_{\boldh_p^{s-2}} + (1 + \eps^{-1})  ( | \lambda|+1 )^{(s-2)/2} \| \bodyg \|_{\boldl^2}) 
\end{equation}
for all $\bodyg \in \boldp^{s-2}$. Here $C>0$ is a constant which is independent of $\lambda$, $\eps$, and $\bodyg$.
\end{thm}

\begin{proof} 
We will now demonstrate that the weak solution provided by Theorem \ref{thm:spectrum_weak} can, in fact, be made into a strong solution of the problem \eqref{eq:lin_prob_div}, \eqref{eq:lin_prob_freesurf}, \eqref{eq:lin_prob_per}, \eqref{eq:stat_lin}. In order to avoid the lack of regularity due to the ``artificial'' corners in our domain, we turn to the equivalent problem of 
finding a weak solution $(\vel_1,q_1)$ of the problem \eqref{eq:lin_prob_div}, \eqref{eq:lin_prob_freesurf}, \eqref{eq:lin_prob_per}, \eqref{eq:stat_lin} on the larger domain $\Omega_1$. By choosing $q_1$ such that $\int_{\Omega_1} q_1 = 0$ we can ensure that $(\vel_1,q_1)$ is simply the periodic extension of $(\vel,q)$ to $\Omega_1$. Then it follows from standard results \cite{lady1} that $(\vel_1,q_1)$ has the additional regularity we seek on compactly contained subsets of $\Omega_1$. To obtain regularity all the way up to the boundary, we follow the approach in \cite{solo2} which is applicable to the boundary provided that it is smooth locally. Thus the pair $(\vel_1,q_1)$ has the desired regularity near $S_F$ up to and including the intersections with $\Gamma_0$, $ \Gamma_{\length}$ since these regions occur on a smooth portion of the free surface on $\Omega_1$. It follows that $\vel_{\per} \in \boldh^2_{\textrm{loc}}(\Omega_{\infty}) $ and $q_{\per} \in H^1_{\textrm{loc}}(\Omega_{\infty})$, hence  $\vel \in \boldp^2$ and $q \in H^1_{\per}$.

To see that $(\vel,q)$ provides us with a strong solution of our problem, we only need to verify that \eqref{eq:lin_prob_freesurf} and \eqref{eq:stat_lin} are satisfied. Using \eqref{eq:weak_form}, for all $\uvel \in \boldp^1$ we have
\begin{equation}
\label{eq:strong_check} (-\mu \lap \vel - \lambda \vel + \grad q - \bodyg, \uvel)_{\boldl^2} = \int_{\partial \Omega} \mathbf{S} (\vel , q) \cdot \overline{\uvel} = \int_{S_F} \mathbf{S} (\vel , q) \cdot \overline{\uvel}. 
\end{equation}
Taking $\uvel \in {}^0 \boldc^{\infty}_{\per \sigma}$ implies that $-\mu \lap \vel - \lambda \vel + \grad q - \bodyg$ lies in the orthogonal complement of $ \overline{ {}^0 \boldc^{\infty}_{{\per} \sigma}}^{\| \cdot \|_{\boldl^2}}$. It was demonstrated in the proof of Proposition \ref{prop:orthog} that this orthogonal complement consists of the gradients of functions in $ H^1_{\per}$, so that $-\mu \lap \vel - \lambda \vel + \grad q - \bodyg = \grad p$ for some $p \in H^1_{\per}$. However, \eqref{eq:weak_form} now yields
\begin{equation}
(\bodyg, \uvel)_{\boldl^2} = (-\mu \lap \vel - \lambda \vel + \grad (q-p) , \uvel)_{\boldl^2} = \langle \vel , \uvel \rangle + \int_{\partial \Omega} \mathbf{S} (\vel , q-p) \cdot \overline{\uvel} - \int_{\Omega} (q-p) \grad \cdot \overline{\uvel}
\end{equation}
for all $\uvel \in \boldh^1_{\per}$. Restricting $\uvel$ to $\boldc^{\infty}_{\textrm{c}}$ and exploiting \eqref{eq:q_weak_form} reduces this to $\int_{\Omega} p \grad \cdot \overline{\uvel} = 0$. Integrating by parts, we see that $\int_{\Omega} \grad p \cdot \overline{\uvel} $ vanishes for arbitrary $\uvel \in \boldc^{\infty}_{\textrm{c}}$. Since this is a dense subset of $\boldl^2$, $\grad p = 0$ and $q$ satisfies \eqref{eq:stat_lin}. All that remains is to show that \eqref{eq:lin_prob_freesurf}, the free surface boundary condition, is also satisfied. From \eqref{eq:strong_check} we now immediately obtain
\begin{equation} \int_{S_F} \mathbf{S} (\vel , q) \cdot \overline{\uvel}=0
\end{equation}
for all $\uvel \in \boldp^1$. Following the lead of \cite{solo1}, we localize to a neighborhood $\Sigma \subset S_F$ and construct $\uvel \in \boldp^1$ such that $\uvel |_{S_F} = ( \mathbf{S}(\vel, q) - (\mathbf{S}(\vel, q) \cdot \mathbf{n}) \mathbf{n} ) \phi$ where $\phi$ is a smooth nonnegative function vanishing outside $\Sigma$. Then
\begin{align}
\int_{S_F} \mathbf{S} (\vel , q) \cdot \overline{\uvel} &= \int_{\Sigma} | \mathbf{S}(\vel, q) - (\mathbf{S}(\vel, q) \cdot \mathbf{n}) \mathbf{n} |^2 \overline{\phi} + (\mathbf{S}(\vel, q) \cdot \mathbf{n}) \mathbf{n} \cdot \overline{( \mathbf{S}(\vel, q) - (\mathbf{S}(\vel, q) \cdot \mathbf{n}) \mathbf{n} ) \phi}\\
&= \int_{\Sigma} | \mathbf{S}_{\textrm{tan}}(\vel) |^2 \overline{\phi} \\
 &=0
\end{align}
implies that $\mathbf{S}_{\textrm{tan}}(\vel) = 0$ on $\Sigma$. Since $\Sigma$ was chosen arbitrarily, we obtain $\mathbf{S}(\vel, q) = (\mathbf{S}(\vel, q) \cdot \mathbf{n}) \mathbf{n} $ on $S_F$. Let $\theta(\vel,q) = q - 2 \mu \rad^{-2} \sum a_i a_j D_j v_i \in H^1_{\per}$. Since $\theta(\vel,q) |_{S_F} = \mathbf{S}(\vel, q) \cdot \mathbf{n}$, \eqref{eq:strong_check} yields
\begin{equation}\int_{\partial \Omega} \theta (\vel , q) \mathbf{n} \cdot \overline{\uvel}=\int_{\Omega}\grad \theta (\vel , q) \cdot \overline{\uvel}=0
\end{equation}
for all $\uvel \in \boldp^1$. By density, $\grad \theta(\vel,q) \in (\boldp^0)^{\perp}$ and $\theta(\vel,q) = p + \omega$ for some $p \in {}^0 H^1_{\per}$ and $\omega \in \real$. Since this implies $\mathbf{S}(\vel, q) \cdot \mathbf{n} = q - 2 \mu \rad^{-2} \sum a_i a_j D_j v_i = \omega$ on $S_F$, we take $q^* = q - \omega$ and obtain a unique strong solution $(\vel,q^*) \in \boldv^2 \times H^1_{\per}$ of the problem \eqref{eq:lin_prob_div}, \eqref{eq:lin_prob_freesurf}, \eqref{eq:lin_prob_per}, \eqref{eq:stat_lin}. 

To further increase regularity, we turn to the standard \emph{a priori} estimates of Agmon, Douglis, and Nirenberg (ADN) \cite{ADN}. Here it is useful to work on $\tor$ rather than $\Omega$. The associated problem is readily seen to be uniformly elliptic in the sense of ADN with boundary conditions satisfying the complementing condition. The \emph{a priori} estimates in $\tor$ then lead to the following estimate in $\Omega$:
\begin{equation}
\| q^* \|_{H^{s-1}_{\per}} + \sum_{j=1}^3 \| v_j \|_{H^s_{\per}} \leq C_{\lambda} \sum_{j=1}^3 \| g_j \|_{H^{s-2}_{\per}} \label{eq:soln_ineq}
\end{equation}
for a positive constant $C_\lambda$ which depends on $\lambda$. 
Thus $(\vel,q^*) \in \boldv^s \times H^{s-1}_{\per}$ is the unique solution of \eqref{eq:lin_prob_div}, \eqref{eq:lin_prob_freesurf}, \eqref{eq:lin_prob_per}, \eqref{eq:stat_lin} and $\sigma(A) \subset \{ \lambda \in \mathbb{C} : \mbox{Re} (\lambda) \geq 0 \}$. Now all that remains is to show that the resolvent estimate \eqref{eq:resolve_ineq} is satisfied. From \eqref{eq:soln_ineq} we obtain the estimate
\begin{equation}
\| \vel \|^2_{\boldh^s_{\per}} \leq \left( \sum_{j=1}^3 \| v_j \|_{H^s_{\per}} \right)^2 \leq 3 C^2_{\lambda} \| \bodyg \|^2_{\boldh_{\per}^{s-2}}=
3 C^2_{\lambda} \|(A - \lambda I)\vel \|^2_{\boldh_{\per}^{s-2}}.
\end{equation}
Thus we have
\begin{align}
\| \vel \|_{\boldh^s_{\per}} & \leq c_1 \| (A+I)\vel \|_{\boldh_{\per}^{s-2}} \\
& \leq c_1 \left( \| (A - \lambda I)\vel \|_{\boldh_{\per}^{s-2}} + ( |\lambda |+1) \| \vel \|_{\boldh_{\per}^{s-2}} \right) \\
& \leq c_2 \left( \| \bodyg \|_{\boldh_{\per}^{s-2}} + ( |\lambda |+1) \| \vel \|^{(s-2)/s}_{\boldh_{\per}^{s}} \| \vel \|^{2/s}_{\boldl^2} \right), \label{eq:lin_soln_est1}
\end{align}
where $c_1$ and $c_2$ are positive constants which do not depend on $\lambda$. Here we have used complex interpolation between $\boldl^2$ and $\boldh^s_{\per}$. Finally, we apply H\"{o}lder to \eqref{eq:sequi_est} which yields 
\begin{equation}
\| \bodyg \|_{\boldl^2}  \geq \frac{| \lambda |}{\sqrt{2}} \| \vel \|_{\boldl^2}. \label{eq:resolve_ineq2}
\end{equation}
Now let us restrict ourselves to $| \lambda | > \eps$ for arbitrary $\eps >0$. If $s=2$, then \eqref{eq:lin_soln_est1} and \eqref{eq:resolve_ineq2} yield \eqref{eq:resolve_ineq} directly. Otherwise, we can apply Young's inequality to \eqref{eq:lin_soln_est1} obtain
\begin{align}
\| \vel \|_{\boldh^s_{\per}} & \leq c_2 \left( \| \bodyg \|_{\boldh_{\per}^{s-2}} + c_3 ( |\lambda |+1)^{s/2} \| \vel \|_{\boldl^2} + \frac{1}{2c_2} \| \vel \|_{\boldh_{\per}^{s}}  \right)\\
& \leq c_4 \left( \| \bodyg \|_{\boldh_{\per}^{s-2}} + ( |\lambda |+1)^{s/2} \| \vel \|_{\boldl^2} \right)\\
& \leq c_5 \left( \| \bodyg \|_{\boldh_{\per}^{s-2}} + ( 1 + \eps^{-1}) ( |\lambda |+1)^{(s-2)/2}  \| \bodyg \|_{\boldl^2} \right),
\end{align}
where $c_3$, $c_4$, and $c_5$ are positive constants which do not depend on $\lambda$ and $\epsilon$. Since $\vel = R(\lambda;A) \bodyg$, this completes the proof.
\end{proof}

Since $\boldh^s_{\per}$ is compactly embedded in $\boldh^{s-2}_{\per}$ for $s\geq 2$, Riesz-Schauder theory and Corollary \ref{0sol} imply the following result. It follows that the kernel of $A$ contains constants only.

\begin{cor}
 $\sigma(A)$ consists of isolated eigenvalues of finite multiplicity. Moreover, the eigenvalue $0$ of $A$ has multiplicity $1$.
\end{cor}

\section{The Inhomogeneous Cauchy Problem} \label{sec:semi}
\setcounter{equation}{0}

We can now show that $-A$ is the infinitesimal generator of an analytic semigroup of contractions. This is the main step involved in constructing solutions to the linear problem \eqref{eq:lin_prob_main}--\eqref{eq:lin_prob_per}. We refer the reader to \cite{engel_nagel} for standard results in semigroup theory.

\begin{thm} \label{thm:semigroup} The operator $-A$, with domain $\boldv^2$, generates an analytic semigroup of contractions, $J(t)$, on $\boldp^0$ with $\| J(t) \| = 1$. 
\end{thm}

\begin{proof} As we seek to apply Lumer-Phillips, we begin by showing that $-A$ is dissipative. To do this, we must improve (slightly) upon the estimate provided by \eqref{eq:resolve_ineq2}. For $\lambda < 0$, we obtain
\begin{equation}
| (\bodyg, \vel)_{\boldl^2} | = | \langle \vel, \vel \rangle | = - \lambda \| \vel \|^2_{\boldl^2} + \frac{\mu}{2} \sum_{i,j=1}^3 \int_{\Omega} | D_j v_i + D_i v_j |^2 \geq - \lambda \| \vel \|^2_{\boldl^2}. \label{eq:semigroup_form}
\end{equation}
Dissipativity now follows using the H\"{o}lder inequality. Since $A+I$ is surjective by Theorem \ref{thm:resolve_inv} and $\boldp^0$ is reflexive (as a Hilbert space), we can apply Lumer-Phillips to obtain that $\boldv^2$ is dense in $\boldp^0$ and $-A$ generates a $C_0$ semigroup of contractions, $J(t)$, on $\boldp^0$. As the generator of a $C_0$ semigroup of contractions, $-A$ is closed (see Theorem II.1.4 in \cite{engel_nagel}, for example) and using \eqref{eq:resolve_ineq2} together with Theorem 12.31 from \cite{renrog} we see that $J(t)$ is actually an analytic semigroup on $\boldp^0$. Now, since $J(t)$ is a semigroup of contractions, we have $\| J(t) \| \leq 1$. However, $0$ is contained in the point spectrum of $-A$ (see the discussion preceding Theorem \ref{thm:resolve_inv}).  Hence $\| J(t) \| = 1$ as required.
\end{proof}

With this semigroup result in hand, we are finally ready to solve the inhomogeneous linear problem \eqref{eq:lin_prob_main}--\eqref{eq:lin_prob_per}. Theorem \ref{thm:semigroup} immediately provides a solution to the Cauchy problem \eqref{eq:var_lin_prob1}--\eqref{eq:var_lin_prob2} and makes the Paley-Wiener theory utilized in \cite{beale1} unnecessary. Here we abbreviate the sets $G = (0,T) \times \Omega$ and $ \partial G_F = (0,T) \times S_F$.

\begin{thm} \label{thm:simp_lin} Let $3<s \leq 4$, $T>0$, and $\bodyf \in \boldk^{s-2}_{\per }$ such that $P\bodyf (0,\cdot)=0$. Then the problem \eqref{eq:lin_prob_main}--\eqref{eq:lin_prob_per} has a unique solution $(\vel,q)$ such that $\vel \in \boldk^s_{\per}$, $\grad q \in \boldk^{s-2}_{\per}$, and $q |_{S_F} \in K^{s-3/2}_{\per}(\partial G_F)$. Moreover, this solution satisfies
\begin{equation}
 \| \vel \|_{\boldk^s_{\per}} + \| \grad q \|_{\boldk^{s-2}_{\per}} + \| q |_{S_F} \|_{K^{s-3/2}_{\per}(\partial G_F)} \leq C \| \bodyf \|_{\boldk^{s-2}_{\per}}, \label{eq:lin_prob_est}
\end{equation}
where $C$ is a positive constant which is independent of $T$ and $\bodyf$.
\end{thm}

\begin{proof} First we notice that $P\bodyf \in C^{0,(s-3)/2} ([0,T]; \boldp^0)$ by the Sobolev Embedding Theorem. Combining 
Corollary 4.3.3 and Theorem 4.3.5(iii) from \cite{pazy}, the abstract Cauchy problem
\begin{align}
\dot{\vel} + A \vel = P \bodyf \\
\vel(0,\cdot) = 0
\end{align}
has a unique strong solution $\vel \in C^{1,(s-3)/2} ([0,T]; \boldp^0)$, with $\vel(t) \in \boldv^2$ for each $t \in [0,T]$. Here we are exploiting the fact that $-A$ is the generator of an analytic semigroup on $\boldp^0$. Note that $\vel$ is a strong solution in the sense of semigroups, i.e.\,, $\vel$ is differentiable almost everywhere on $[0,T]$, with $\dot{\vel} \in L^1((0,T);\boldp^0)$, such that $\vel(0,\cdot) = 0$ and $\dot{\vel}(t) = -A \vel(t) + P \bodyf (t) $ almost everywhere on $[0,T]$. In fact, $\vel$ is a classical solution in the semigroup sense since it is continuously differentiable with respect to time.

To show that $\vel \in \boldk^s_{\per}$, we reconsider the abstract Cauchy problem (now with a new unknown variable $\tilde{\vel}$) from another perspective. We begin by applying the periodic analog of Lemma 2.2 from \cite{beale1} in order to extend $P \bodyf$ to $ \boldk^{s-2}_{\per}(\real \times \Omega)$ in such a way that the extension is bounded independent of $T$ and vanishes for $t<0$. Multiplying through the abstract Cauchy problem by the weight $w(t) = e^{-t}$ and taking Fourier transforms in $t$, we obtain
\begin{equation}
\mathcal{F}_w(\tilde{\vel})(\xi) = (A+(1+i\xi)I)^{-1}\mathcal{F}_w (P\bodyf)(\xi).
\end{equation}
Since it is clear that $\mathcal{F}_w(P \bodyf)(\xi) \in \boldp^{s-2}$, this uniquely defines $\mathcal{F}_w(\tilde{\vel})(\xi) \in \boldv^s$ by Theorem \ref{thm:resolve_inv}. Making use of the Fourier transform characterization of $H^s$-spaces for $s \in \real^+$ (e.g., see \cite{adams}) and the fact that Fourier transforms are unitary transformations, we have
\begin{align}
 \| \tilde{\vel} \|^2_{\boldk^s_{\per}(\real \times \Omega)} &\leq 2 \left( \| \tilde{\vel} \|^2_{\boldl^2(\real; H^s_{\per})} + \| \tilde{\vel} \|^2_{\boldh^{s/2}(\real; L^2)} \right)\\
 & = 2 \left( \| \mathcal{F}_w(\tilde{\vel})(\xi+1) \|^2_{\boldl^2(\real; H^s_{\per})} + \| (1 + \xi^2)^{s/4} \mathcal{F}_w(\tilde{\vel})(\xi+1) \|^2_{\boldl^2(\real; L^2)} \right)\\
  & = 2 \int_{\real} \left( \| \mathcal{F}_w(\tilde{\vel})(\xi+1) \|^2_{\boldh^s_{\per}} + (1 + \xi^2)^{s/2} \| \mathcal{F}_w(\tilde{\vel})(\xi+1) \|^2_{\boldl^2} \right) \, d \xi.
\end{align}

Applying the resolvent estimate \eqref{eq:resolve_ineq} to the first term of the integral, we obtain
\begin{align}
\| \mathcal{F}_w(\tilde{\vel})(\xi+1) \|^2_{\boldh^s_{\per}} & \leq c_1 \left( \| \mathcal{F}_w(P \bodyf)(\xi+1) \|_{\boldh^{s-2}_{\per}} + 2 (|1+i(\xi+1)|+1)^{(s-2)/2} \| \mathcal{F}_w(P \bodyf)(\xi+1) \|_{\boldl^2} \right)^2\\
& \leq c_2 \left( \| \mathcal{F}_w(P \bodyf)(\xi+1) \|^2_{\boldh^{s-2}_{\per}} + (\sqrt{1+(\xi+1)^2}+1)^{s-2} \| \mathcal{F}_w(P \bodyf)(\xi+1) \|^2_{\boldl^2} \right)\\
& \leq c_2 \left( \| \mathcal{F}_w(P \bodyf)(\xi+1) \|^2_{\boldh^{s-2}_{\per}} + \left( 3 \sqrt{1+\xi^2} \right)^{s-2} \| \mathcal{F}_w(P \bodyf)(\xi+1) \|^2_{\boldl^2} \right)\\
& \leq c_3 \left( \| \mathcal{F}_w(P \bodyf)(\xi+1) \|^2_{\boldh^{s-2}_{\per}} + \left( 1+\xi^2 \right)^{(s-2)/2} \| \mathcal{F}_w(P \bodyf)(\xi+1) \|^2_{\boldl^2} \right)
\end{align}
where $c_1$, $c_2$, and $c_3$ are positive constants which are independent of $\xi$ and $\bodyf$. Similarly, we can apply estimate \eqref{eq:resolve_ineq2} to the second term of the integral to get
\begin{align}
(1 + \xi^2)^{s/2} \| \mathcal{F}_w(\tilde{\vel})(\xi+1) \|^2_{\boldl^2} & \leq 2 (1 + \xi^2)^{s/2} |1 + i(\xi+1) |^{-2} \| \mathcal{F}_w(P \bodyf)(\xi+1) \|^2_{\boldl^2} \\
& = 2 \left( \frac{1 + \xi^2}{1+(1+\xi)^2} \right) (1 + \xi^2)^{(s-2)/2} \| \mathcal{F}_w(P \bodyf)(\xi+1) \|^2_{\boldl^2}\\
& \leq 6 (1 + \xi^2)^{(s-2)/2} \| \mathcal{F}_w(P \bodyf)(\xi+1) \|^2_{\boldl^2}.
\end{align}
Combining these estimates yields 
\begin{align}
\| \tilde{\vel} \|^2_{\boldk^s_p(\real \times \Omega)} & \leq c_4 \int_{\real} \| \mathcal{F}_w(P \bodyf)(\xi+1) \|^2_{\boldh^{s-2}_{\per}} + \left( 1+\xi^2 \right)^{(s-2)/2} \| \mathcal{F}_w(P \bodyf)(\xi+1) \|^2_{\boldl^2} \, d \xi \\
& = c_4 \left( \| \mathcal{F}_w(P \bodyf)(\xi+1) \|^2_{\boldl^2(\real; H^{s-2}_{\per})} + \| (1 + \xi^2)^{(s-2)/4} \mathcal{F}_w(P \bodyf)(\xi+1) \|^2_{\boldl^2(\real; L^2)} \right)\\
& = c_4 \left( \| P \bodyf \|^2_{\boldl^2(\real; H^{s-2}_{\per})} + \| P \bodyf \|^2_{\boldh^{(s-2)/2}(\real; L^2)} \right)\\
& \leq c_4 \| P \bodyf \|^2_{\boldk^{s-2}_{\per}(\real \times \Omega)},
\end{align}
where $c_4>0$ is a constant which is independent of $\xi$ and $\bodyf$. By uniqueness, we must have $\vel = \tilde{\vel}|_{G} \in \boldk^s_{\per}$. We now seek a suitable $q$ so that $(\vel,q)$ is the unique solution of \eqref{eq:lin_prob_main}--\eqref{eq:lin_prob_per}. For fixed $t$, this amounts to finding a unique $q \in H_{\per}^{s-1}$ such that
\begin{align}
\grad q &= \mu \lap \vel + A \vel + \bodyf - P \bodyf & & \mbox{on } \Omega\\
q &= Q \vel & & \mbox{on } S_F.
\end{align}
Since $s>3$, this is easily accomplished by taking the divergence of the first equation and applying Proposition \ref{prop:uniquesoln}. All that remains is to show that $(\vel,q )$ satisfies \eqref{eq:lin_prob_est}. To estimate $q$ we first notice that
\begin{equation}
\grad q = \mu (I-P) \lap \vel + \grad Q \vel + (I-P)\bodyf.
\end{equation}
The only term which we do not yet know how to estimate is $\grad Q \vel$. However, since $\lap Q \vel = 0$ on $\Omega$ and $Q \vel = \phi$ on $S_F$ where $\phi = 2 \mu \rad^{-2} \sum_{i,j=1}^2 a_i a_j D_j v_i \in H^{s-1}_{\per}$, it follows from Proposition \ref{prop:P}(3) that $\grad Q \vel = P( \grad \phi)$. Then by Proposition \ref{prop:P}(2),
\begin{equation}
\| \grad Q \vel \|_{\boldk^{s-2}_{\per}} = \| P( \grad \phi) \|_{\boldk^{s-2}_{\per}} \leq c_5 \| \grad \phi \|_{\boldk^{s-2}_{\per}} \leq c_6 \| \vel \|_{\boldk^s_{\per}}, 
\end{equation} 
where $c_5$ and $c_6$ are positive constants. Similarly, since $Q$ was constructed so that $q = Q \vel$ on $S_F$,
\begin{equation}
\| q|_{S_F} \|_{K^{s-3/2}_{\per}(\partial G_F)} = \| Q \vel |_{S_F} \|_{K^{s-3/2}_{\per}(\partial G_F)} \leq c_7 \| Q \vel \|_{K^{s-1}_{\per}} \leq c_8 \| \vel \|_{\boldk^s_{\per}},
\end{equation}
where $c_7$ and $c_8$ are positive constants. Thus, combining estimates, we obtain
\begin{align}
\| \vel \|_{\boldk^s_{\per}} + \| \grad q \|_{\boldk^{s-2}_{\per}} + \| q|_{S_F} \|_{K^{s-3/2}_{\per}(\partial G_F)} & \leq c_9 \left( \| \vel \|_{\boldk^s_{\per}} + \| \bodyf \|_{\boldk^{s-2}_{\per}} \right)\\
& \leq c_9 \left( \| \tilde{\vel} \|_{\boldk^s_{\per}(\real \times \Omega)} + \| \bodyf \|_{\boldk^{s-2}_{\per}} \right)\\
& \leq c_{10} \left( \| \mathbf{h} \|_{\boldk^{s-2}_{\per}(\real \times \Omega)} + \| \bodyf \|_{\boldk^{s-2}_{\per}} \right)\\
& \leq c_{11} \| \bodyf \|_{\boldk^{s-2}_{\per}}
\end{align}
where $c_9$, $c_{10}$, and $c_{11}$ are positive constants which do not depend on $\bodyf$ (or $T$).
\end{proof}

\section{Concluding Remarks}

With Theorem \ref{thm:simp_lin} in hand, it is now straightforward to follow through the fixed point approach outlined in \cite{beale1}, with minor revisions, to obtain the following local existence result. For details about these modifications we refer the reader to \cite{Ceci}.

\begin{thm} \label{thm:NLP} Suppose $3<s<\frac{7}{2}$. For any $\uvel_0 \in \boldv^{s-1}$ there exists $T>0$, depending on $\| \uvel_0 \|_{\boldh^{s-1}_{\per}}$, so that the problem \eqref{eq:lag1}--\eqref{eq:lag7} has a solution $(\vel,q)$ with $\vel \in \boldk^s_{\per}, q \in K^{s-3/2}_{\per}(\partial G_F)$, and $\grad q \in \boldk^{s-2}_{\per}$. 
\end{thm}

Since the same arguments can be successfully applied to yield a similar local existence result when the initial displacement is taken to be nonzero in \eqref{eq:lag7}, uniqueness of solutions can then be shown to follow in the standard way. It is also a simple matter to prove that the unique solution given by Theorem \ref{thm:NLP} is axisymmetric provided that $\uvel_0$ is taken to be axisymmetric. \cite{Ceci} contains additional details.

\end{document}